\documentclass[12pt,a4paper,reqno,parskip=full]{amsart}
\usepackage{amsmath}
\usepackage{amsfonts}
\usepackage{amssymb}
\usepackage{colonequals}
\usepackage{parskip}
\usepackage{tikz}
\usepackage{tikz-cd}
\usetikzlibrary{arrows,shapes,positioning}
\usetikzlibrary{decorations.markings}
\tikzstyle arrowstyle=[scale=1]
\tikzstyle directed=[postaction={decorate,decoration={markings,
    mark=at position .58 with {\arrow[arrowstyle]{to}}}}]

\begingroup
\makeatletter
\@for\theoremstyle:=definition,remark,plain\do{
\expandafter\g@addto@macro\csname th@\theoremstyle\endcsname{
\addtolength\thm@preskip\parskip}}
\endgroup

\numberwithin{equation}{section}
\theoremstyle{plain}

\newtheorem{theorem}[subsection]{Theorem}
\newtheorem{proposition}[subsection]{Proposition}
\newtheorem{lemma}[subsection]{Lemma}
\newtheorem{corollary}[subsection]{Corollary}

\newtheorem{question}[subsection]{Question}
\newtheorem{remark}[subsection]{Remark}

\theoremstyle{definition}

\newtheorem{definition}[subsection]{Definition}
\newtheorem{example}[subsection]{Example}

\renewcommand{\leq}{\leqslant}
\renewcommand{\geq}{\geqslant}
\newcommand{\eps}{\varepsilon}

\DeclareMathOperator{\hor}{hor}

\DeclareMathOperator{\rad}{rad}

\def\A{{\mathcal A}}
\def\C{{\mathcal C}}
\def\D{{\mathcal D}}
\def\Z{{\mathbb Z}}
\def\E{{\mathbb E}}
\def\F{{\mathbb F}}
\def\H{{\mathbb H}}
\def\FF{{\mathcal F}}
\def\R{{\mathbb R}}
\def\N{{\mathbb N}}
\def\Q{{\mathbb Q}}
\def\m{{\mathfrak m}}
\def\S{{\mathbb S}}

\def\vs{\vspace{11pt}}
\def\ni{\noindent}

\begin{document}

\title{The boundary action of a sofic random subgroup of the free group}

\author{Jan Cannizzo}

\address{Stevens Institute of Technology\newline
\indent Department of Mathematical Sciences\newline
\indent Castle Point on Hudson\newline
\indent Hoboken, NJ 07030\newline
}

\email{jan.c.cannizzo@gmail.com}

\thanks{I am very grateful to my advisor, Vadim Kaimanovich, for his advice and support, and for encouraging me to work on the subject of this paper. I thank the anonymous referee for suggesting several improvements.}

%\subjclass{}

\maketitle

\begin{abstract}
We prove that the boundary action of a sofic random subgroup of a finitely generated free group is conservative (there are no wandering sets). This addresses a question asked by Grigorchuk, Kaimanovich, and Nagnibeda, who studied the boundary actions of individual subgroups of the free group. We also investigate the cogrowth and various limit sets associated to sofic random subgroups. We make heavy use of the correspondence between subgroups and their Schreier graphs, and central to our approach is an investigation of the asymptotic density of a given set inside of large neighborhoods of the root of a sofic random Schreier graph.
\end{abstract}

\section{Introduction}

The study of \emph{invariant random subgroups}, meaning subgroups of a given group whose distribution is conjugation-invariant, has recently attracted a lot of attention. Vershik has called for a description of all nonatomic conjugation-invariant measures on the lattice of subgroups of a given countable group \cite{V1} and provided such a description in the case of the infinite symmetric group \cite{V2}. Such measures naturally arise from the boundary actions of self-similar groups, such as the \emph{Basilica group} (see the treatment of D'Angeli, Donno, Matter, and Nagnibeda \cite{ADMN}), or the famous Grigorchuk group (see \cite{Vo}), and progress has recently been made in understanding the spaces of invariant random subgroups of free groups \cite{B} and lamplighter groups \cite{BGK}. Ab\'{e}rt, Glasner, and Vir\'{a}g recently generalized \emph{Kesten's theorem} to invariant random subgroups \cite{AGV}. Moreover, invariant random subgroups are closely connected with the theory of sofic groups (see, for example, the survey \cite{Pe}) and sofic equivalence relations \cite{EL}.

There is a fruitful interplay between groups and graphs, as is evidenced, for instance, in the classic paper of Stallings \cite{S}. Central to our approach is the fact that it is possible to switch back and forth between subgroups and their \emph{Schreier graphs} (objects which generalize Cayley graphs), allowing one to think about subgroups in geometric terms. Accordingly, the study of invariant random subgroups is tantamount to the study of invariant random Schreier graphs (which in turn belongs to the theory of discrete measured equivalence relations established by Feldman and Moore~\cite{FM}).

Intuitively speaking, invariant random Schreier graphs behave rather like Cayley graphs, the analogy being that, whereas a Cayley graph is spatially homogenous, insofar as it is vertex-transitive, i.e.\ invariant upon shifting the root, an invariant random Schreier graph is \emph{stochastically homogenous} (see \cite{K2} for the origin of the term), insofar as its distribution is invariant upon shifting the root. Grigorchuk, Kaimanovich, and Nagnibeda \cite{GKN} recently studied the ergodic properties of the action of a subgroup $H\leq\F_n$ of a finitely generated free group on the boundary $\partial\F_n$ equipped with the uniform measure (a situation which is analogous to the action of a Fuchsian group on the boundary of the hyperbolic plane equipped with Lebesgue measure). In particular, they used Schreier graphs to describe the \emph{Hopf decomposition} (into conservative and dissipative parts) of this action. Although the boundary action of an arbitrary subgroup may be conservative, dissipative, or such that both its conservative and dissipative parts have positive measure \cite{GKN}, the boundary action of a normal subgroup is necessarily conservative, as follows from \cite{K1}. Our main result is an extension of this result to \emph{sofic random subgroups}. That is, we show that \emph{the boundary action of a sofic random subgroup of a finitely generated free group is conservative} (Theorem~\ref{conservative}), addressing a question asked in \cite{GKN}.

Before proving our main result, we undertake an investigation of the asymptotic density of a given set inside of a random invariant Schreier graph. Our main question of interest (Question~\ref{q}) can be formulated as follows: given a nontrivial subset $A$ of the space of Schreier graphs, must the density of $A$ inside of large neighborhoods of the root of an invariant random Schreier graph be bounded away from zero? If so, then we say the invariant random Schreier graph has \emph{property D}. A positive answer to the question would amount to a new \emph{ergodic theorem} for invariant random graphs (see \cite{BN} for an overview of many ergodic theorems). Unfortunately, we are unable to answer Question~\ref{q}, but by introducing a notion which we call \emph{relative thinness} and assuming that our invariant random Schreier graph $\Gamma$ is sofic, we are able to show that $\Gamma$ fails to satisfy the aforementioned property only if its geometry is quite peculiar (Proposition~\ref{lop}), a fact which allows us to prove Theorem~\ref{conservative}.

The paper is organized as follows: In Section~2, we give an introduction to Schreier graphs and invariant random Schreier graphs, making plain their connection with subgroups. Section~3 is devoted to making precise the question of whether a given set is asymptotically dense inside of large neighborhods of the root of an invariant random Schreier graphs. In Section~4, we introduce sofic invariant subgroups and thereafter, in Section~5, the notion of relative thinness, which allows us to shed some light on Question~\ref{q}. In Section~6, we prove our main result, showing that the boundary action of a sofic random subgroup is conservative (Theorem~\ref{conservative}). Finally, in Section~7, we tease out several consequences of Theorem~\ref{conservative}, namely a bound on the \emph{cogrowth} of a sofic random subgroup (Corollary~\ref{cogrowth}) and a theorem on the size of various \emph{limit sets} associated to sofic random subgroups of $\F_n$ (Theorem~\ref{lim}). We also give examples showing that the \emph{radial limit set} may have full or zero measure, thus completely characterizing the possible measures of the limit sets of a sofic invariant subgroup.

\section{The space of Schreier graphs of a countable group}

Given a countable group $G$ with generating set $\A=\{a_i\}_{i\in I}$ and a subgroup $H\leq G$, consider the natural action of $G$ on the space of (right) cosets $G/H$. This action is transitive and determines a graph $\Gamma=(\Gamma,H)$ as follows. The vertex set of $\Gamma$ is identified with $G/H$, and two vertices $Hg$ and $Hg'$ are connected with an edge directed from $Hg$ to $Hg'$ and labeled with the generator $a_i$ if and only if $Hga_i=Hg'$. The graph $\Gamma$ (which is \emph{rooted} at $H$, meaning that we distinguish the vertex $H$) is called a (right) \emph{Schreier graph}, and we denote by $\Lambda(G)$ the space of (isomorphism classes) of (right) Schreier graphs of $G$, where two Schreier graphs are said to be isomorphic if there exists a graph isomorphism between them which preserves the edge-labeling and root. Note that Schreier graphs are necessarily $2|\A|$-regular, meaning that each of their vertices has degree $2|\A|$ (the degree of a vertex may be defined as the sum of the number of incoming edges and the number of outgoing edges attached to it). Schreier graphs may have both loops (cycles of length one) and multi-edges (multiple edges that join the same pair of vertices). Note also that Schreier graphs naturally generalize Cayley graphs, which arise whenever the subgroup $H$ is normal, i.e.\ when the cosets $Hg$ correspond to the elements of a group.

Let us immediately turn our attention to the space of Schreier graphs of the finitely generated free group of rank $n$ with a fixed set of generators, i.e.\
\[
\F_n=\langle a_1,\ldots,a_n\rangle.
\]
This is natural since, as we will presently make clear, every Schreier graph is a Schreier graph of a free group. Our first observation is this: Given a Schreier graph $(\Gamma,H)\in\Lambda(\F_n)$, the subgroup $H\leq\F_n$ can be recovered from $\Gamma$ in a very natural way. Namely, $H$ is precisely the fundamental group $\pi_1(\Gamma,H)$, i.e.\ the set of words read upon traversing closed paths that begin and end at the coset $H$. Note that we thereby identify $\pi_1(\Gamma,H)$ with a specific subgroup of $\F_n$ and are not interested merely in its isomorphism class. By the above discussion, it follows that $\Lambda(G)\subseteq\Lambda(\F_n)$ whenever $G$ is a group with generating set $\A=\{a_1,\ldots,a_n\}$. It also follows that we could define Schreier graphs ``abstractly,'' without appealing to the coset structure determined by a subgroup of $\F_n$. That is, we could define a Schreier graph to be a (connected and rooted) $2n$-regular graph whose edges come in $n$ different colors and are colored so that every vertex is attached to precisely one incoming edge of a given color and one outgoing edge of that color.

There is a natural one-to-one correspondence between the lattice of subgroups of $\F_n$, denoted $L(\F_n)$, and the space of Schreier graphs $\Lambda(\F_n)$. Every subgroup $H\in L(\F_n)$ determines a Schreier graph, and every Schreier graph $\Gamma\in\Lambda(\F_n)$ determines a subgroup of $\F_n$ (by passing to the fundamental group):
\begin{equation*}
\begin{tikzcd}[column sep=55]
L(\F_n)\arrow[bend left]{r}{(\Gamma,H)}
&\Lambda(\F_n)\arrow[bend left]{l}{\pi_1(\Gamma)}\,.
\end{tikzcd}
\end{equation*}
\ni The space of Schreier graphs $\Lambda(\F_n)$ has a natural projective structure. Denote by $\Lambda_r(\F_n)$ the set of (isomorphism classes of) $r$-neighborhoods centered at the roots of elements of $\Lambda(\F_n)$, where by an $r$-neighborhood we mean the subgraph of a Schreier graph induced by the set of vertices at distance less than or equal to $r$ from the root. Then $\Lambda(\F_n)$ may be realized as the projective limit
\begin{equation*}\label{proj}
\Lambda(\F_n)=\varprojlim\Lambda_r(\F_n),
\end{equation*}
where the connecting morphisms $\pi_r:\Lambda_{r+1}(\F_n)\to\Lambda_r(\F_n)$ are restriction maps that send an $(r+1)$-neighborhood $V$ to the $r$-neighborhood $U$ of its root. (Looking at things the other way around, $\pi_r(V)=U$ only if there exists an embedding $U\hookrightarrow V$ that sends the root of $U$ to the root of $V$.) By endowing each of the sets $\Lambda_r(\F_n)$ with the discrete topology, we turn $\Lambda(\F_n)$ into a compact Polish space.

Throughout this paper, we will think of an $r$-neighborhood $U\in\Lambda_r(\F_n)$ both as a rooted graph and as the \emph{cylinder set}
\[
U=\{(\Gamma,x)\in\Lambda(\F_n)\mid U_r(x)\cong U\},
\]
where $U_r(x)$ denotes the $r$-neighborhood of the vertex $x$. Note that a finite Borel measure $\mu$ on $\Lambda(\F_n)$ is the same thing as a family of measures $\mu_r:\Lambda_r(\F_n)\to\R$ that satisfies
\[
\mu_r(U)=\sum_{V\in\pi_r^{-1}(U)}\mu_{r+1}(V)
\]
for all $U\in\Lambda_r(\F_n)$ and for all $r$. As is customary when working with measure spaces, all statements regarding measurable sets will be understood to be valid \emph{modulo zero}, i.e.\ up to the inclusion or exclusion of null sets (in particular, we will avoid use of qualifying expressions such as ``almost every.'')

By an \emph{invariant random subgroup} of a countable group $G$, we will mean a probability measure on $L(G)$ that is conjugation-invariant, i.e.\ invariant under the action $G\circlearrowright L(G)$ given by $(g,H)\mapsto gHg^{-1}$. Via the correspondence between $L(G)$ and $\Lambda(G)$ (indeed, it is via this correspondence that we endow $L(G)$ with its Borel structure), this determines a continuous action on $\Lambda(G)$ which is easily visualized as follows: Given a Schreier graph $(\Gamma,H)$ and an element $g\in G$, where we assume that $g$ has a fixed presentation in terms of the generators of $G$, it is possible to read the element $g$ starting from the root $H$ (or, indeed, from any other vertex). This is accomplished by following, in the proper order, edges labeled with the generators that comprise $g$ (note that following a generator $a_i^{-1}$ is tantamount to traversing a directed edge labeled with $a_i$ in the direction opposite to which the edge is pointing). Applying the group element $g$ to the graph $(\Gamma,H)$ then amounts simply to ``shifting the root" of $(\Gamma,H)$ in the way just described. That is, one begins at the vertex $H$, then follows the path corresponding to the element $g$, and then declares its endpoint to be the new root. Note that if $G$ has generators of order two, then a path corresponding to an element $g\in G$ may not be unique; nevertheless, the endpoint of any path which represents $g$ is uniquely determined by $g$.
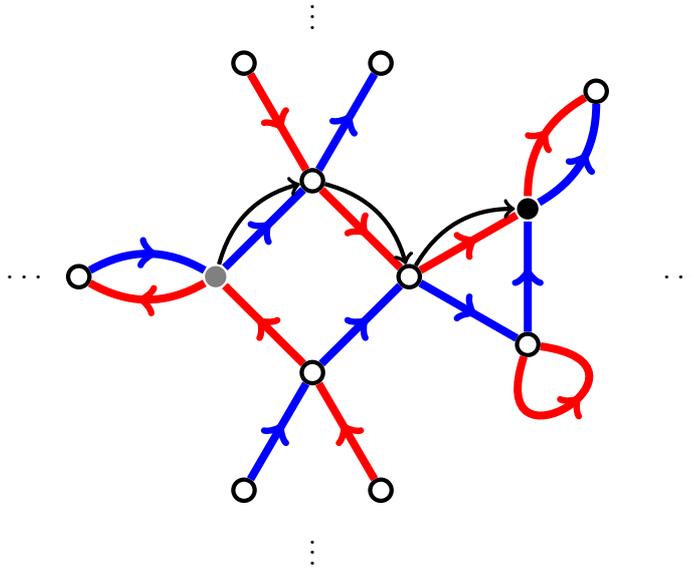
\begin{figure}\label{fig1}
\begin{tikzpicture}[scale=1.8,root node/.style={circle,ultra thick,draw=none,fill,black,inner sep=0pt,minimum size=8pt},
oldroot node/.style={circle,ultra thick,draw=none,fill,gray,inner sep=0pt,minimum size=8pt},
normal node/.style={circle,ultra thick,draw=black,inner sep=0pt,minimum size=8pt}]

%\draw[help lines](-1,-2) grid(4,3);

\node[oldroot node](e) at (1,0){};
\node[normal node](a) at (0,0){};
\node[normal node](b) at (1.707,0.707){};
\node[normal node](a^-1) at (1.707,-0.707){};
\node[normal node](ba) at (2.414,0){};
\node[root node](baa) at (3.28,0.5){};
\node[normal node](bab) at (3.28,-0.5){};
\node[normal node](bb) at (2.207,1.573){};
\node[normal node](ba^-1) at (1.207,1.573){};
\node[normal node](a^-2) at (2.207,-1.573){};
\node[normal node](a^-1b^-1) at (1.207,-1.573){};
\node[normal node](baaa) at (3.78,1.366){};

\foreach \from\to in {a^-1/e,a^-2/a^-1,ba^-1/b,b/ba,ba/baa}
\draw[red,line width=3,directed](\from)--(\to);

\foreach \from\to in {e/b,b/bb,a^-1b^-1/a^-1,a^-1/ba,ba/bab,bab/baa}
\draw[blue,line width=3,directed](\from)--(\to);

\draw[red,line width=3,directed](e) to[bend left](a);
\draw[blue,line width=3,directed](a) to[bend left](e);
\draw[red,line width=3,directed](baa) to[bend left](baaa);
\draw[blue,line width=3,directed](baa) to[bend right](baaa);

\draw[red,line width=3,distance=1cm,directed](bab) to[out=-108,in=-8,](bab);

\node at (4.4,0){$\ldots$};
\node at (-0.4,0){$\ldots$};
\node at (1.707,1.973){$\vdots$};
\node at (1.707,-1.973){$\vdots$};

\draw[->,ultra thick](e) to[bend left](b);
\draw[->,ultra thick](b) to[bend left](ba);
\draw[->,ultra thick](ba) to[bend left](baa);
\end{tikzpicture}
\caption{A Schreier graph of the free group $\F_2=\langle a,b\rangle$, with red edges representing the generator $a$ and blue edges the generator $b$. Shown here is conjugation by the element $ba^2\in\F_2$, which entails starting at the root (the gray vertex), then following the edges corresponding to the generators $b$, $a$, and $a$ again (in that order), and declaring their endpoint to be the new root (the black vertex).}
\end{figure}

The image of a $G$-invariant measure under the identification $H\mapsto(\Gamma,H)$ is a $G$-invariant measure on $\Lambda(G)$ (and hence, via the inclusion $\Lambda(G)\hookrightarrow\Lambda(\F_n)$, an $\F_n$-invariant measure on $\Lambda(\F_n)$). We may thus speak of an \emph{invariant random Schreier graph}. In fact, throughout the remainder of the paper we will treat invariant random subgroups and invariant random Schreier graphs as the same objects, using whichever terminology is more appropriate to the context.

The most basic examples of invariant random Schreier graphs are Dirac measures supported on Cayley graphs: indeed, Cayley graphs (equivalently, Schreier graphs of normal subgroups of $\F_n$) are invariant under conjugation essentially by definition and may be regarded as $1$-periodic points in $\Lambda(\F_n)$. More generally, the uniform measure supported on a finite Schreier graph (of cardinality $k$, say) is an invariant measure, thus giving rise to $k$-periodic points, and it is not difficult to construct examples of infinite periodic Schreier graphs. Of greater interest is the space of \emph{nonatomic} invariant measures, typically supported on aperiodic Schreier graphs. Such measures have recently been the focus of a great deal of research (see \cite{AGV}, \cite{ADMN}, \cite{B}, \cite{BGK}, \cite{V1}, \cite{V2}, and \cite{Vo}), but much remains unknown.

\section{A question regarding the density of sets inside of large neighborhoods}

Our main result is that the boundary action of a sofic random subgroup is conservative. By a theorem of Grigorchuk, Kaimanovich, and Nagnibeda~\cite{GKN}, this assertion is equivalent to the assertion that
\begin{equation}\label{thm}
\lim_{r\to\infty}\frac{|U_r(\Gamma,H)|}{|U_r(\F_n,e)|}=0,
\end{equation}
where the numerator of the above fraction is the size of the $r$-neighborhood of the root of our random Schreier graph and the denominator is the size of the $r$-neighborhood of the identity of the Cayley graph of $\F_n$. In proving this result, our focus will first be on a considerably more general question regarding the asymptotic density of a given set inside of neighborhoods centered at the root of a random graph. This latter question can be formulated as follows: if $A\subseteq(\Lambda(\F_n),\mu)$ is a measurable subset of the space of Schreier graphs and $\mu$ is an invariant measure, then how dense is $A$ inside of $\mu$-random $r$-neighborhoods $U_r(x)\in\Lambda_r(\F_n)$? An informal---and imprecise---way to say what we mean by the density of $A$ in $U_r(x)$ is in terms the function $\rho_{A,r}:\Lambda(\F_n)\to\Q$ given by
\[
\rho_{A,r}(\Gamma,x)\colonequals\frac{|A\cap U_r(x)|}{|U_r(x)|}.
\]
To make this rigorous, note that for any $A\subseteq\Lambda(\F_n)$ there is an induced Borel embedding
\[
\Theta_A:\Lambda(\F_n)\to\bigcup_{\Gamma\in\Lambda(\F_n)}\{0,1\}^\Gamma\equalscolon\{0,1\}^{\Lambda(\F_n)}
\]
which sends a Schreier graph $\Gamma$ to the \emph{binary field} $\FF:\Gamma\to \{0,1\}$ given by
\begin{equation}\label{char}
\FF(x)=\left\{
\begin{array}{lr}
1,&(\Gamma,x)\in A\\
0,&(\Gamma,x)\notin A
\end{array}
\right.,
\end{equation}
where $(\Gamma,x)$ is the Schreier graph obtained from $\Gamma$ by \emph{rerooting} $\Gamma$ at the vertex $x$. The resulting space of binary configurations over elements of $\Lambda(\F_n)$ (namely, the image of $\Theta_A$) serves to ``highlight'' the set $A$, and the corresponding functions $\rho_{A,r}$ may now be written as
\begin{equation}\label{density}
\rho_r(\FF)=\frac{1}{|U_r(x)|}\sum_{y\in U_r(x)}\FF(y).
\end{equation}
Note that if $\mu$ is an invariant measure on $\Lambda(\F_n)$, then $(\Theta_A)_*\mu$ is an invariant measure on $\{0,1\}^{\Lambda(\F_n)}$. From now on, when talking about the density of a given set $A$ inside of $r$-neighborhoods, we will refer to the functions $\rho_{A,r}$ defined over the binary field constructed as per (\ref{char}), without necessarily making mention of the map $\Theta_A$. We are now ready to formulate our question:
\begin{question}\label{q}
Let $\mu$ be an invariant random Schreier graph and $A\subseteq\Lambda(\F_n)$ a Borel set, and consider the average densities
\[
\E(\rho_{A,r})=\int\rho_{A,r}\,d\mu.
\]
Then supposing $\E(\rho_{A,0})>0$, what can be said of the averages $\E(\rho_{A,r})$? Do they converge? Are they bounded away from zero?
\end{question}
More generally, consider an $\F_n$-invariant measure $\mu$ on $\{0,1\}^{\Lambda(\F_n)}$ (which needn't necessarily come from a Borel set $A\subseteq\Lambda(\F_n)$ as above). The following example shows that if such an invariant random binary field has a ``fixed geometry,'' meaning that it is supported on a common underlying graph, then it must answer Question~\ref{q} in the positive.
\begin{example}
Let $\Gamma\in\Lambda(\F_n)$ be a Cayley graph, i.e.\ the Schreier graph of a normal subgroup of $\F_n$, and $\mu$ an invariant measure on $\{0,1\}^\Gamma$. Then one readily verifies that the average densities $\E(\rho_r)$ are all the same. Indeed, we have
\begin{align*}
\int\rho_r\,d\mu&=\int\left(\frac{1}{|U_r(x)|}\sum_{y\in U_r(x)}\FF(y)\right)d\mu\\
&=\frac{1}{|U_r(x)|}\sum_{y\in U_r(x)}\left(\int\FF(y)\,d\mu\right)\\
&=\frac{1}{|U_r(x)|}\sum_{y\in U_r(x)}\E(\rho_0)=\E(\rho_0).
\end{align*}
If, however, our random invariant Schreier graph ceases to be so nice (e.g.\ if it ceases to be vertex-transitive), then the averages $\E(\rho_r)$ can be expected to vary considerably from $\E(\rho_0)$. Our question is: How much? Can they be arbitrarily close to zero if $\E(\rho_0)$ is not zero?
\end{example}
Question~\ref{q} asks whether invariance implies that, given a subset of the space of Schreier of positive measure (in other words, a nontrivial property of the root of our random graph), it will be asymptotically dense inside of large $r$-neighborhoods. For the sake of brevity, let us give this property a name.
\begin{definition}(Property $D$)
We say that an invariant random Schreier graph $\mu$ has \emph{property D} if it answers Question~\ref{q} in the positive, in the sense that, if $A\subseteq\Lambda(\F_n)$ is any subset with $\mu(A)>0$, then the average densities $\E(\rho_{A,r})$ of $A$ inside of $r$-neighborhoods are bounded away from zero.
\end{definition}
We will show that, upon placing a mild condition on our random invariant Schreier graph $\Gamma$---namely \emph{soficity}---the averages $\E(\rho_{A,r})$ can get arbitrarily small only if $\Gamma$ exhibits a rather ``wild'' geometry. To be a little more precise, we will introduce a notion which we call \emph{relative thinness} and show that the average densities $\E(\rho_{A,r})$ can get arbitrarily small only if $\Gamma$ is arbitrarily relatively thin at different scales. We are then able to deduce the conservativity of the boundary action of a sofic random subgroup via the following argument:
\begin{itemize}
\item[i.] If $\Gamma$ satisfies property $D$ (and is not the Dirac measure concentrated on the Cayley graph of $\F_n$, a case which is easily dealt with), then there exists a number $k\in\N$ such that the set of Schreier graphs whose roots belong to a cycle of length $k$ has positive measure, and whose density inside of $r$-neighborhoods is therefore bounded away from zero. The fact that cycles of bounded length are sufficiently dense inside of $\Gamma$ is in turn enough for us to show that $\Gamma$ must satisfy (\ref{thm}).

\item[ii.] If $\Gamma$ does not satisfy property $D$, then its geometry is such that it cannot grow too quickly; in particular, we are again able to show that $\Gamma$ must satisfy the condition (\ref{thm}).
\end{itemize}
It is worth pointing out that, if our property $D$ held for all invariant random Schreier graphs, then the argument of Theorem~\ref{conservative} would imply that the boundary action of any invariant random subgroup (sofic or not) of $\F_n$ is conservative.

\section{Sofic invariant subgroups}

The class of \emph{sofic groups}, first defined by Gromov~\cite{G} and given their name by Weiss~\cite{W}, is a large class of groups which has recently received a great deal of attention. Roughly speaking, a finitely generated group is sofic if its Cayley graph can be approximated by a sequence of finite Schreier graphs. Amenable groups (for which \emph{F\o lner sequences} determine approximating sequences) and residually finite groups (for which finite quotients serve as approximating sequences) are immediate examples of sofic groups. In fact, so large is the class of sofic groups that it is unknown whether all groups are sofic. For more on sofic groups, we refer the reader to the survey of Pestov~\cite{Pe}.

The notion of soficity, which can be formulated in terms of the weak convergence of measures, naturally generalizes to objects other than groups, such as \emph{unimodular random graphs}---see, for instance, \cite{AL}. In another context, Elek and Lippner~\cite{EL} have recently defined soficity for \emph{discrete measured equivalence relations}, a setting which subsumes invariant random Schreier graphs. To make sense of the definition, observe that the uniform probability measure on a finite Schreier graph $\Gamma$ determines an invariant measure on $\Lambda(\F_n)$, namely the uniform measure supported on the conjugacy class of the associated subgroup $\pi_1(\Gamma)$. The definition now goes as follows:
\begin{definition}\label{sofic}(Sofic random Schreier graph)
An invariant random Schreier graph $\mu$ is \emph{sofic} if there exists a sequence of finite Schreier graphs $\{\Gamma_i\}_{i\in\N}$ such that $\mu_i\to\mu$ weakly, where $\mu_i$ is the invariant measure on $\Lambda(\F_n)$ determined by $\Gamma_i$.
\end{definition}
The convergence of which we speak also goes under the name of \emph{Benjamini-Schramm convergence}, after the paper \cite{BS}. We note that (as is also done in \cite{BS}) the weak convergence of measures in Definition~\ref{sofic} can be thought of in more geometric terms as follows: Suppose first that $\Gamma$ is a Cayley graph (which becomes an invariant random Schreier graph when identified with the Dirac measure concentrated on itself). We say that a finite graph $(\Gamma',\mu)$ equipped with the uniform probability measure is an \emph{$(r,\eps)$-approximation} to $\Gamma$ if there exists a set $A\subseteq\Gamma'$ of measure $\mu(A)>1-\eps$ such that for all $x\in A$, the $r$-neighborhood of $x$ in $\Gamma'$ is isomorphic (in the category of edge-labeled graphs) to the $r$-neighborhood of the identity (or, indeed, of any other vertex) in $\Gamma$. The graph $\Gamma$ is sofic precisely if it admits an $(r,\eps)$-approximation for any pair $(r,\eps)$, where $r\in\N$ and $\eps>0$. A group $G$ is thus sofic if, given any Cayley graph $\Gamma$ of $G$, it is possible to construct finite graphs which locally look like $\Gamma$ at almost all of their points. More generally, suppose that $\Gamma$ is a random invariant Schreier graph. The distribution of $\Gamma$ naturally determines a probability measure $\mu_r$ on $\Lambda_r(\F_n)$, the set of $r$-neighborhoods of Schreier graphs of $\F_n$, and we again say that a finite graph $(\Gamma,\mu)$ equipped with the uniform probability measure is an $(r,\eps)$-approximation to $\Gamma$ if for all $U\in\Lambda_r(\F_n)$ we have $|\mu(U)-\mu_r(U)|<\eps$. Then, as before, a random invariant Schreier graph is sofic precisely if it admits finite $(r,\eps)$-approximations for any pair $(r,\eps)$.

Our definition does not take exactly the same form as the ones given, for instance, in \cite{EL} or \cite{G}. The main difference is that we require our approximating sequence to consist of bona fide Schreier graphs, and not, as is usually the case, of graphs which need not have the structure of a Schreier graph at all of their points. Let us therefore quickly show that our definition---which we feel is a bit cleaner---is in fact equivalent to the usual one.
\begin{theorem}
If there exist finite graphs $(\Gamma_i,\mu_i)$ which are a sofic approximation to $\mu$, then they may be modified to create finite Schreier graphs $(\Gamma_i',\mu_i')$ which are a sofic approximation to $\mu$.
\end{theorem}

\begin{remark}
Here the graphs $\Gamma_i$ need not have the structure of a Schreier graph at each of their points, i.e.\ there may exist points whose degree is not $2n$ or are such that the edges attached to them do not have a Schreier labeling. Another caveat that should be pointed out is that a Schreier graph is by definition connected and rooted, although we do not actually impose these conditions in Definition~\ref{sofic} or the above proposition: there is no sense in assigning a root to the graphs of a sofic approximation (as every vertex is effectively treated as a root), and it is often natural for such graphs to have several connected components (e.g.\ if the measure they approximate is supported on a set of several distinct Cayley graphs).
\end{remark}

\begin{proof}
Let $\Gamma_i$ be an $(r,\eps)$-approximation to $\mu$ and $A\subseteq\Gamma_i$ the set of points at which $\Gamma_i$ does not have the structure of a Schreier graph. Let $\Gamma_i'$ be the subgraph of $\Gamma_i$ induced by the set $\Gamma_i\backslash A$ and $A'\subseteq\Gamma_i'$ the set of points at which $\Gamma_i'$ does not have the structure of a Schreier graph. Note that $A'$ is a subset of the set of neighbors of the removed set $A$, and that therefore $\mu_i(A\cup A')<\eps$ (since the $r$-neighborhood $U_r(x)\subseteq\Gamma_i$ of any point $x\in A$ does not approximate $\mu$, neither does  the $r$-neighborhood of any neighbor of $x$, provided $r>1$).

Now, the edges attached to points $x\in A'$ are properly labeled with the generators $a_1,\ldots,a_n$ of $\F_n$---the only problem is that some generators may be missing, i.e.\ it may be that $\deg(x)<2n$. We thus ``stitch up'' the graph $\Gamma_i'$ as follows: for every generator $a_i$ which does not label any of the edges (neither incoming nor outgoing) attached to a given point $x\in A'$, add a loop to $x$ and label it with $a_i$. If, on the other hand, there exists precisely one edge (assume without loss of generality that it is outgoing) attached to $x$ and labeled with a generator $a_i$, then consider the longest path $\gamma$ whose edges are labeled only with $a_i$ and which is attached to $x$. The endpoint of $\gamma$ will be a vertex $y\in A'$ distinct from $x$; to ``complete the cycle,'' we thus need only join $x$ and $y$ with an edge and label this edge with $a_i$ in the obvious way. By repeating this procedure for every vertex in $A'$, we ensure that $\Gamma_i'$ has the structure of a Schreier graph at every point while modifying it only on a set of very small measure. It follows that the sequence of Schreier graphs $(\Gamma_i',\mu_i')$ is a sofic approximation to $\mu$.
\end{proof}

Note that Definition~\ref{sofic} readily generalizes to invariant random fields: one must simply define convergence with respect to finite $\{0,1\}$-labeled Schreier graphs. We will make use of the following lemma later.
\begin{lemma}\label{soficfield}
Let $\mu$ be a sofic random Schreier graph and $A\subseteq\Lambda(\F_n)$ a Borel set. Then the invariant random field $(\Theta_A)_*\mu$ is also sofic.
\end{lemma}
\begin{proof}
Denote by $A_r$ the collection of cylinder sets $U\in\Lambda_r(\F_n)$ such that $\mu(A\cap U)>0$. Clearly, $A\subseteq A_r$, and moreover $\mu(A_r\backslash A)\equalscolon\eps_r\to0$, i.e.\ the sets $A_r$ approximate $A$. Let $\Gamma$ be a finite $(r,\eps)$-approximation to $\mu$, and construct a binary field $\FF:\Gamma\to\{0,1\}$ by assigning to a given vertex $x\in\Gamma$ the value $1$ if the cylinder set corresponding to its $r$-neighborhood $U_r(x)$ belongs to $A_r$ and the value $0$ otherwise. Then $\FF$ is an $(r,\eps)$-approximation to $\mu_r$ and hence an $(r,\eps+\eps_r)$-approximation to $\mu$. By constructing fields $\FF_i$ in this way for a sequence of finite graphs $\Gamma_i$ which are $(r_i,\eps_i)$-approximations to $\mu$, with $r_i\to\infty$ and $\eps_i\to0$, we obtain a sofic approximation to $(\Theta_A)_*\mu$.
\end{proof}
Morally speaking, Lemma~\ref{soficfield} allows us to phrase Question~\ref{q} in terms of finite graphs, namely those which come from a sofic approximation. Working with finite graphs in turn has several advantages, as we show in the next section.

\section{Relative thinness}

\ni In order to investigate Question~\ref{q}, we would like to introduce a notion which we call \emph{relative thinness}. To be more precise, let $\Gamma$ be a Schreier graph, and consider the functions $\tau_r:\Gamma\to\Q$ defined by
\[
\tau_r(x)\colonequals\sum_{y\in U_r(x)}\frac{1}{|U_r(y)|}.
\]
Note that if, say, all of the $r$-neighborhoods of $\Gamma$ have the same size (as is the case, for instance, when $\Gamma$ is a Cayley graph), then $\tau_r\equiv1$. If, on the other hand, the $r$-neighborhood of a point $x\in\Gamma$ is small compared to the $r$-neighborhoods near it, then one will have $\tau_r(x)<1$ (and if it is large compared to the $r$-neighborhoods near it, then one will have $\tau_r(x)>1$). We thus say that a Schreier graph $\Gamma$ is \emph{relatively thin at scale $r$} at a point $x\in\Gamma$ if $\tau_r(x)<1$ (if a piece of cloth is worn down at a particular spot, then the regions surrounding that spot will have more mass than is to be found at the spot itself).

One feature of relative thinness is that it is ``tempered,'' meaning that if $\Gamma$ is very thin at $x$ and $y$ is a neighbor of $x$, then $\Gamma$ will be thin at $y$ as well. To be more precise, let us say that a function $f:\Gamma\to\R$ is \emph{$C$-Lipschitz} if whenever $x,y\in\Gamma$ are neighbors,
\[
f(x)\leq Cf(y)
\]
for some constant $C\geq1$. Likewise, we say that a family of functions $\{f_i:\Gamma_i\to\R\}_{i\in I}$ is \emph{uniformly $C$-Lipschitz} over the family of graphs $\{\Gamma_i\}_{i\in\N}$ if each $f_i$ is $C$-Lipschitz for some constant $C\geq1$ that does not depend on $i$. We now have the following lemma.
\begin{lemma}\label{lip}
Let $\Gamma\in\Lambda$ be a Schreier graph of $\F_n$. Then there exists a constant $C\geq1$ such that the family of functions $\{\tau_r\}_{r\in\N}$ is uniformly $C$-Lipschitz over $\Gamma$.
\end{lemma}
\begin{proof}
Note first that if $x$ and $y$ are neighbors in $\Gamma$, then we have the bound
\begin{align}\label{nbrsize}
|U_r(x)|\geq\frac{1}{2n-1}|U_r(y)|.
\end{align}
Put $S\colonequals U_r(y)\backslash U_r(x)$, and let $S'$ denote a choice, for each vertex $z\in S$, of a neighbor $z'$ which belongs to $U_r(x)$. Then
\begin{align*}
\tau_r(y)-\tau_r(x)&\leq\sum_{z\in S}\frac{1}{|U_r(z)|}\\
&\leq(2n-1)^2\sum_{z\in S'}\frac{1}{|U_r(z)|}\\
&\leq(2n-1)^2\tau_r(x).
\end{align*}
Here the second line is obtained by applying the inequality (\ref{nbrsize}) and using the fact that points in $S'$ may have at most $2n-1$ neighbors in $S$. It follows that each $\tau_r$ is $C$-Lipschitz with $C=(2n-1)^2+1$.
\end{proof}
Moreover, it turns out that, at least in the model case of a finite Schreier graph (which carries a unique invariant probability measure), thinness and the densities $\rho_{A,r}$ given by (\ref{density}) are directly related to one another.
\begin{proposition}\label{prop1}
Let $(\Gamma,A,\mu)$ be a finite Schreier graph $\Gamma$ equipped with the uniform probability measure, together with a subset $A\subseteq\Gamma$. Then
\[
\int_\Gamma\rho_{A,r}\,d\mu=\int_A\tau_r\,d\mu,
\]
where $\rho_{A,r}$ is the $r$-neighborhood density of the set $A$.
\end{proposition}
\begin{proof}
One must simply observe that, whether summing $\rho_{A,r}$ over $\Gamma$ or $\tau_r$ over $A$, for a given point $x\in\Gamma$ the quantity $1/|U_r(x)|$ is summed exactly once for every point $y\in A$ such that $x\in U_r(y)$.
\end{proof}
\ni As a corollary, we obtain:
\begin{corollary}\label{one}
Given a finite Schreier graph $(\Gamma,\mu)$ equipped with the uniform probability measure, $\tau_r$ integrates to one over $\Gamma$.
\end{corollary}
\begin{proof}
Simply choose $A=\Gamma$ in the hypotheses of Proposition~\ref{prop1}. Then $\rho_{A,r}\equiv1$, so that we have
\begin{align*}
\int_\Gamma\tau_r\,d\mu=\int_\Gamma\rho_{A,r}\,d\mu=\int_\Gamma1\,d\mu=1.\tag*{\qedhere}
\end{align*}
\end{proof}
We thus find that the ``average thinness'' of a finite Schreier graph is always one. Proposition~\ref{prop1} can therefore be interpreted as saying that, if the average of $\rho_{A,r}$ over a finite Schreier graph $\Gamma$ is small relative to $\E(\rho_{A,0})=\mu(A)$, then the set $A$ must be concentrated at points where $\Gamma$ is relatively thin (at scale $r$).

Corollary~\ref{one} tells us that, if $\Gamma$ is a finite Schreier graph, then by integrating the functions $\tau_r$ against the uniform probability measure on $\Gamma$, we obtain a new probability measure $\nu_r$. Suppose now that $\mu$ is a sofic random Schreier graph, and let $\{\Gamma_i\}_{i\in\N}$ be a sofic approximation to $\mu$. Then one readily verifies that the sequence of probability measures $\nu_{r,i}$---those obtained by integrating $\tau_r$ against the uniform measures $\mu_i$---converges weakly to a probability measure $\nu_r$ on $\Lambda(\F_n)$. That is, soficity implies that $\tau_r$ is a density with respect to $\mu$.
\begin{proposition}\label{ergseq}
Let $\mu$ be a sofic random Schreier graph which is ergodic and which does not satisfy property $D$. Then there exist finite Schreier graphs $(\Gamma_i,A_i,\mu_i)$ together with subsets $A_i\subseteq\Gamma_i$ such that the $\Gamma_i$ are a sofic approximation to $\mu$, $\mu_i(A_i)\to1$, and $\E(\tau_i\mid A_i)\to0$.
\end{proposition}
\begin{proof}
If $\mu$ does not satisfy property $D$, then there exists a set $A\subseteq\Lambda(\F_n)$ with $\mu(A)>0$ such that $\E(\rho_{A,r})\to0$ along some subsequence of radii $r\in\N$, and hence such that $\E(\tau_r\mid A)\to0$. Let $\{g_i\}_{i\in\N}$ be an enumeration of $\F_n$ (e.g.\ the lexicographic order), and put
\[
A_k\colonequals A\cup g_1A\cup\ldots\cup g_kA.
\]
It follows from the fact that the $\tau_r$ are uniformly $C$-Lipschitz (Proposition~\ref{lip}) that $\E(\tau_r\mid A_k)\to0$ for any $k$. Indeed, putting 
\[
m\colonequals\max_{1\leq i\leq k}|g_i|,
\]
we have $\E(\tau_r\mid A_k)\leq C^m\E(\tau_r\mid A)\to0$. Moreover, by ergodicity, $\mu(A_k)\to1$. By Lemma~\ref{soficfield}, there exists a sofic approximation $\{\FF_{i,k}\}_{i\in\N}$ for each invariant random field $(\Theta_{A_k})_*$, which is the same thing as a sequence of finite Schreier graphs $(\Gamma_{i,k},A_{i,k},\mu_{i,k})$ such that the $A_{i,k}$ approximate $A_k$ (just take $A_{i,k}=\{x\in\Gamma_{i,k}\mid\FF_{i,k}(x)=1\}$). By choosing an appropriate diagonal sequence, we prove our claim.
\end{proof}

Suppose again that $\mu$ is a sofic random Schreier graph which is ergodic and does not satisfy property $D$. Our next goal is to show that the geometry of $\mu$ must be quite peculiar. To do so, we will look at the sofic approximation to $\mu$ guaranteed by Proposition~\ref{ergseq}, i.e.\ the sequence of finite Schreier graphs $(\Gamma_i,A_i,\mu_i)$, with $\mu_i(A_i)\to1$ and $\E(\tau_i\mid A_i)\to0$. A trick we will emply is the following: instead of working with the functions $\tau_r$ and letting $r$ vary, we may instead modify the structure of our Schreier graphs and work only with the function $\tau_1$. Thus if $\Gamma_i$ is one of our Schreier graphs (constructed, by default, with respect to the standard generating set $\A=\{a_1,\ldots,a_n$\}), denote by $\Gamma_i^{(r)}$ what we call the \emph{$r$-contraction of $\Gamma_i$} obtained by regarding it as a Schreier graph of $\F_n$ constructed with respect to the generating set consisting of all group elements of length less than or equal to $r$. One readily verifies that $\tau_r$ over $\Gamma$ agrees with $\tau_1$ over $\Gamma^{(r)}$, in the sense that the diagram
\begin{equation*}
\begin{tikzcd}
\Gamma_i\arrow[hookrightarrow]{r}\arrow{dr}[swap]{\tau_r}
&\Gamma_i^{(r)}\arrow{d}{\tau_1}\\
&\Q
\end{tikzcd}
\end{equation*}
commutes (here the upper arrow is the obvious identification between the vertices of $\Gamma_i$ and the vertices of $\Gamma_i^{(r)}$). By modifying the structure of our graphs in this way (for ever larger values of $r$) and choosing an appropriate diagonal sequence, our sofic approximation now takes the form of a sequence of finite Schreier graphs $(\Gamma_i,A_i,\mu_i)$ such that $\mu_i(A_i)\to1$ and $\E(\tau_1\mid A_i)\to0$.

We do not know of any invariant random Schreier graph which fails to have property $D$. In order to get a sense of what a sequence of graphs satisfying the aforementioned conditions might look like, however, consider the following example.
\begin{example}\label{bip}
Let $X_N$ be a set of $2^N$ points and $Y_N$ a set of $N$ points, and let $\Gamma_N$ denote the complete bipartite graph between $X_N$ and $Y_N$, i.e.\ the graph obtained by adding to the set $X_N\sqcup Y_N$ all possible edges $(x,y)$ such that $x\in X_N$ and $y\in Y_N$. Then the sequence of graphs $(\Gamma_N,X_N,\mu_N)$ has the property that $\E(\tau_1\mid X_N)\to0$. Indeed, it is easy to see that for fixed $N$, $\tau_1$ is constant over each of $X_N$ and $Y_N$, and that $\left.\tau_1\right|_{X_N}\to0$ whereas $\left.\tau_1\right|_{Y_N}\to\infty$. At the same time, we have $\mu_N(X_N)\to1$.
\end{example}
Note, however, that the graphs constructed in Example~\ref{bip} cannot be realized as a sequence of contracted Schreier graphs. Indeed, suppose that, possibly upon adding loops to the vertices of the bipartite graphs of Example~\ref{bip} and turning some of their edges into multi-edges, we were able to label their edges with generators of $\F_n$. Then for each vertex $x\in X_N$, it must be the case that one of its ``external edges,'' meaning an edge $(x,y)$ with $y\in Y_N$, is labeled with one of the standard generators $a_1,\ldots,a_n$ (or one of their inverses)---were this not the case, $x$ would be fixed by every $a_i$ and hence by $\F_n$ itself, a contradiction, since $x$ has $\F_n$-labeled external edges attached to it. By the pigeonhole principle, there must thus exist a generator $a_i^{\pm1}$ and a subset $X_N'\subseteq X_N$ of measure $\mu_N(X_N')\geq\mu_N(X_N)/2n$ such that $a_i^{\pm1}X_N'\subseteq Y_N$. But this is again a contradiction, since $\mu_N(Y_N)\to0$ and $\mu_N$ is an invariant measure. Alternatively, note that there is an ever widening gap between the values of $\tau_1$ over $X_N$ and $Y_N$, which violates the fact that $\tau_1$ is $C$-Lipschitz (Proposition~\ref{lip}).

The family of graphs constructed in Example~\ref{bip} has what one might call a ``lopsided structure.'' That is to say, graphs in the family split into a set of large measure and a set of small measure in such a way that all of the neighbors of a given vertex in the large set belong to the small set. The next proposition shows that, despite the fact that the bipartite graphs considered above cannot be realized as Schreier graphs, a version of this phenomenon must occur whenever $\mu$ is a sofic random Schreier graph which is ergodic and does not satisfy property $D$ (see also Figure~2).
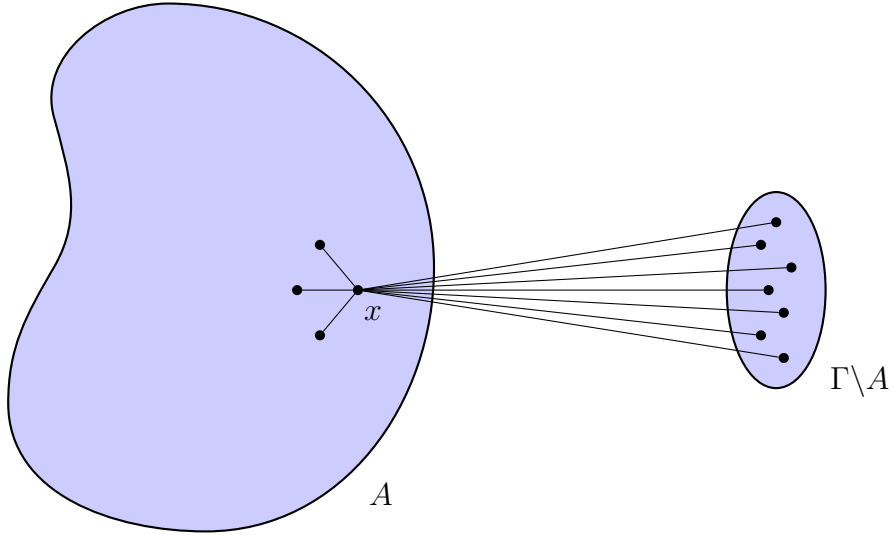
\begin{figure}\label{fig2}
\centering
\begin{tikzpicture}[main node/.style={fill,circle,thick,draw=black,inner sep=0pt,minimum size=3pt}]
%\draw[help lines] (-8,-5) grid(8,5);
%\node at (0,0){$o$};

\draw[thick,fill=blue!20] (-4.5,3.5) to [out=0,in=90] (-1,0) to [out=-90,in=0] (-4,-3.5)
to [out=180,in=270] (-6.6,-1.8) to [out=90,in=240] (-6,0)
to [out=60,in=-75] (-6,2) to [out=105,in=180] (-4.5,3.5);

\draw[thick,fill=blue!20] (3.5,-0.3) ellipse (0.65cm and 1.3cm);

\node at (-1.7,-3){$A$};
\node at (4.6,-1.5){$\Gamma\backslash A$};
\node at (-1.8,-0.6){$x$};

\node[main node](x) at (-2,-0.3){};
\node[main node](r) at (-2.5,0.3){};
\node[main node](s) at (-2.8,-0.3){};
\node[main node](t) at (-2.5,-0.9){};
\node[main node](1) at (3.5,0.6){};
\node[main node](2) at (3.3,0.3){};
\node[main node](3) at (3.7,0){};
\node[main node](4) at (3.4,-0.3){};
\node[main node](5) at (3.6,-0.6){};
\node[main node](6) at (3.3,-0.9){};
\node[main node](7) at (3.6,-1.2){};

\foreach \from/\to in {x/1,x/2,x/3,x/4,x/5,x/6,x/7,x/r,x/s,x/t}
\draw(\from)--(\to);
\end{tikzpicture}
\caption{Finite (contracted) Schreier graphs $\Gamma$ that approximate invariant random subgroups which do not satisfy property $D$ have a subset $A\subseteq\Gamma$ of large measure such that, for a random point $x\in A$, the large majority of its neighbors belong to the complement $\Gamma\backslash A$.}
\end{figure}
\begin{proposition}\label{lop}
Let $\mu$ be a sofic random Schreier graph which is ergodic and does not satisfy property $D$. Then there exists a sequence of finite (contracted) Schreier graphs $(\Gamma_i,A_i,\mu_i)$ such that the $\Gamma_i$ are a sofic approximation to $\mu$, $\mu_i(A_i)\to1$, and
\[
\lim_{i\to\infty}\E\left(\left.\frac{\deg_{A_i}(x)}{\deg_{\Gamma_i\backslash A_i}(x)}\,\right| A_i\right)=0,
\]
where $\deg_A(x)$ denotes the number of neighbors of $x$ in the set $A$.
\end{proposition}
\begin{proof}
Let $(\Gamma_i,A_i,\mu_i)$ be finite contracted Schreier graphs that are a sofic approximation to $\mu$ and such that $\mu_i(A_i)\to1$ and $\E(\tau_1\mid A_i)\to0$. We have
\begin{align*}
\E(\tau_1\mid A_i)&=\frac{1}{|A_i|}\sum_{x\in A_i}\frac{1+\deg_{A_i}(x)}{1+\deg(x)}\\
&=\frac{1}{|A_i|}\sum_{x\in A_i}\frac{1+\deg_{A_i}(x)}{1+\deg_{A_i}(x)+\deg_{\Gamma_i\backslash A_i}(x)}<\eps_i,
\end{align*}
with $\eps_i\to0$. It follows that, for any $K>0$, the subsets $A_{i,K}\subset A_i$ over which $\deg_{\Gamma_i\backslash A_i}(x)\leq K\deg_{A_i}(x)$ satisfy $\mu_i(A_{i,K})\to0$. Indeed, were this not the case, we would have
\begin{align*}
\E(\tau_1\mid A_i)&\geq\E(\tau_1\mid A_{i,K})\mu_i(A_{i,K})\\
&\geq\frac{1}{|A_{i,K}|}\sum_{x\in A_{i,K}}\frac{1+\deg_{A_i}(x)}{1+(K+1)\deg_{A_i}(x)}\,\mu_i(A_{i,K})\\
&\geq\frac{\delta}{K+1}
\end{align*}
for all $i\in\N$, where $\delta>0$ is a fixed lower bound of the values $\mu_i(A_{i,K})$. We therefore find that the ratio of the expected number of internal neighbors to external neighbors of points in $A_i$ tends to zero, as desired.
\end{proof}

\section{Conservativity of the boundary action}

There is a natural \emph{boundary}, denoted $\partial\F_n$, associated to the free group $\F_n=\langle a_1,\ldots,a_n\rangle$, and it admits a number of interpretations. Viewing elements of $\F_n$ as finite reduced words in the alphabet $\A^\pm=\{a_1^{\pm1},\ldots,a_n^{\pm1}\}$, the boundary $\partial\F_n$ is the space of infinite reduced words in the alphabet $\A^\pm$ endowed with the topology of pointwise convergence. Equivalently, $\partial\F_n$ is the projective limit of the spheres $\partial U_r(\F_n,e)$, i.e.\ the sets of words in $\F_n$ of length $r$, where each such set is given the discrete topology and the connecting maps serve to delete the last symbol of a given word (the space $\partial\F_n$ is thus a Cantor set provided $n>1$). Taking a more geometric view, $\partial\F_n$ is naturally homeomorphic to the \emph{space of ends} of the Cayley graph of $\F_n$. The latter object being a Gromov hyperbolic space, $\partial\F_n$ may be viewed as the \emph{hyperbolic boundary} of $\F_n$ (so that $\F_n\cup\partial\F_n$ is its \emph{hyperbolic compactification}). And when equipped with the uniform measure $\m$ (which we will define in a moment), ($\partial\F_n,\m)$ is naturally isomorphic to the \emph{Poisson boundary} of the simple random walk on $\F_n$, a fact first established by Dynkin and Malyutov \cite{DM}.

Grigorchuk, Kaimanovich, and Nagnibeda \cite{GKN} recently studied the ergodic properties of the action of a subgroup $H\leq\F_n$ on the boundary of $\F_n$ equipped with the uniform measure $\m$. To be explicit, $\m$ is the probability measure given by
\begin{equation}\label{meas}
\m(g)=\frac{1}{2n(2n-1)^{|g|-1}},
\end{equation}
where we again allow $g$ to represent both an element of $\F_n$ (here $|g|$ is the length of $g$) and the cylinder set consisting of those infinite words whose truncations to their first $|g|$ symbols are equal to $g$. Of course, the denominator of (\ref{meas}) is just the cardinality of the sphere $\partial U_{|g|}(\F_n,e)$.

The aforementioned boundary action, which we denote by $H\circlearrowright(\partial\F_n,\m)$, is analogous to the action of a Fuchsian group on the boundary of the hyperbolic plane $\partial\H^2\cong\S^1$ equipped with Lebesgue measure: both actions, the latter being a classical object of study, are boundary actions of discrete groups of isometries of a Gromov hyperbolic space. In \cite{GKN}, the combinatorial structure of the space $\F_n$, and especially the Schreier graphs corresponding to its subgroups, are exploited in order to investigate the action $H\circlearrowright(\partial\F_n,\m)$. In particular, Theorem~2.12 of \cite{GKN} gives a combinatorial characterization of the Hopf decomposition of this action. Let us review this result.

Let $G\circlearrowright(X,\mu)$ be a quasi-invariant action of a countable group on a Lebesgue space, i.e.\ a measure space whose nonatomic part is isomorphic to the unit interval equipped with Lebesgue measure. Recall that such an action is \emph{conservative} if every measurable subset $E\subseteq X$ is \emph{recurrent}, meaning that it is contained in the union of its $g$-translates, where $g\in G\backslash\{e\}$. The action is \emph{dissipative} if $(X,\mu)$ is the union of the translates of a \emph{wandering set}, i.e.\ a subset $E\subseteq X$ whose $G$-translates are pairwise disjoint. Every quasi-invariant action $G\circlearrowright(X,\mu)$ admits a unique \emph{Hopf decomposition}
\[
X=\C\sqcup\D
\]
into conservative and dissipative parts (see \cite{A} and the references therein), so that the action of $G$ restricted to $\C$ is conservative and the action of $G$ restriced to $\D$ is dissipative.

Turning our attention to the action $H\circlearrowright(\partial\F_n,\m)$, consider the Schreier graph $(\Gamma,H)$ of $H$, and let $T\subseteq\Gamma$ be a \emph{geodesic spanning tree}, i.e.\ a spanning tree such that $d_T(H,Hg)=d_\Gamma(H,Hg)$ for all vertices (cosets) $Hg$. Such a spanning tree always exists. Let $\Omega_H\subseteq\partial\F_n$ denote the \emph{Schreier limit set}. It is the set of infinite words (which of course correspond to infinite paths in $\Gamma$) that pass through edges not in $T$ infinitely often. Let $\Delta_H\subseteq\F_n$ denote the \emph{Schreier fundamental domain}. It is the set of infinite words that remain in $T$. We then have the following boundary decomposition:
\begin{equation}\label{decomp}
\partial\F_n=\Omega_H\sqcup\bigsqcup_{h\in H}h\Delta_H.
\end{equation}
That is, $\partial\F_n$ is the disjoint union of the Schreier limit set and the $H$-translates of the Schreier fundamental domain. It is shown in \cite{GKN} (see Theorem~2.12) that the decomposition (\ref{decomp}) is in fact the Hopf decomposition of the action $H\circlearrowright(\partial\F_n,\m)$.
\begin{theorem}\rm(Grigorchuk, Kaimanovich, and Nagnibeda)
\emph{The conservative part of the boundary action $H\circlearrowright(\partial\F_n,\m)$ coincides with the Schreier limit set $\Omega_H$. The dissipative part coincides with the $H$-translates of the Schreier fundamental domain $\Delta_H$.}
\end{theorem}
Moreover, Theorem~4.10 of \cite{GKN} shows that the measure of the Schreier fundamental domain is related to the growth of the Schreier graph $(\Gamma,H)$ of $H$.
\begin{theorem}\rm{(Grigorchuk, Kaimanovich, and Nagnibeda)}\label{diss}
\emph{The measure of the Schreier fundamental domain determined by a proper subgroup $H\in L(\F_n)$ is equal to}
\[
\m(\Delta_H)=\lim_{r\to\infty}\frac{|\partial U_r(\Gamma,H)|}{|\partial U_r(\F_n,e)|},
\]
\emph{and the above sequence of ratios is nonincreasing.}
\end{theorem}
\begin{remark}
Note that Theorem~\ref{diss} remains valid if one replaces the spheres $\partial U_r$ with neighborhoods $U_r$.
\end{remark}
The action $H\circlearrowright(\partial\F_n,\m)$ may be conservative. This is the case, for example, whenever $H$ is of finite index, or when $H$ is a normal subgroup of $\F_n$. The action may also be dissipative, which is the case, for instance, whenever $H$ is finitely generated and of infinite index. It may also be the case that both the conservative and dissipative parts of the action have positive measure: see, for instance, Example~4.27 of \cite{GKN}. It is our aim, however, to show that the boundary action of an invariant random subgroup is necessarily conservative. To this end, let us understand a \emph{$k$-cycle} to be a closed path which is isomorphic to a $k$-sided polygon. Our main idea is that an invariant random Schreier graph which satisfies property $D$  must have a certain ``density of $k$-cycles,'' i.e.\ that there exists a $k$ such that a given vertex of an invariant random Schreier graph belongs to a $k$-cycle with positive probability, and that this in turn restricts the growth of our random graph enough to render $\Delta_H$ a null set.
\begin{theorem}\label{conservative}
The boundary action $H\circlearrowright(\partial\F_n,\m)$ of a sofic random subgroup of the free group is conservative.
\end{theorem}
\begin{proof}
Suppose first that $\mu$ is an invariant random Schreier graph that satisfies property $D$. It is not difficult to see that, with the exception of one trivial case, there must always exist a number $k$ such that the Borel set $A$ of Schreier graphs whose roots belong to a $k$-cycle has positive measure. Indeed, if this were not the case, then $\mu$ would be the Dirac measure concentrated on the Cayley graph of $\F_n$ (whose boundary action is of course conservative). By assumption, there thus exists an $\eps>0$ such that $\E(\rho_{A,r})\geq\eps$ for all $r$. Put $f(r)\colonequals2n(2n-1)^{r-1}$, let $X_r$ denote the size of the radius-$r$ sphere centered at the root of a $\mu$-random Schreier graph, let $\ell=\lfloor k/2\rfloor$, and let $r\geq1$ be an initial radius. Trivially, $\E(X_r)\leq f(r)$. We are then able to bound $\E(X_{r+\ell})$ as
\[
\E(X_{r+\ell})\leq f(r+\ell)-\eps f(r)
\]
and, continuing inductively, to obtain the general bound
\begin{align}
\E(X_{r+(m+1)\ell})&\leq(2n-1)^\ell\E(X_{r+m\ell})-\eps\left(\E(X_{r+m\ell})-\eps\E(X_{r+(m-1)\ell})\right)\nonumber\\
&=\left((2n-1)^\ell-\eps\right)\E(X_{r+m\ell})+\eps^2\E(X_{r+(m-1)\ell})\label{recur},
\end{align}
since each $k$-cycle that passes through the boundary of an $(r+m\ell)$-neighborhood allows us to decrease the trivial bound on the size of the boundary of an $(r+(m+1)\ell)$-neighborhood by one. Note that (\ref{recur}) is a linear homogenous recurrence relation with characteristic polynomial
\[
\chi(t)=t^2-\left((2n-1)^\ell-\eps\right)t-\eps^2.
\]
It is easy to see that $\chi$ has distinct real roots. The general solution of the recurrence relation (\ref{recur}) thus yields the bound
\begin{equation}\label{bound}
\begin{split}
\E(X_{r+m\ell})\leq C_0&\left((2n-1)^\ell-\eps+\sqrt{\left((2n-1)^\ell-\eps\right)^2+4\eps^2}\right)^m\\
&+C_1\left((2n-1)^\ell-\eps-\sqrt{\left((2n-1)^\ell-\eps\right)^2+4\eps^2}\right)^m,
\end{split}
\end{equation}
whereupon applying initial conditions readily gives $C_0=C_1=f(r)/2$ (in order to simplify notation, we have doubled the roots of $\chi$). By Theorem~\ref{diss}, we have
\begin{align*}
\E(\m(\Delta_H))&=\int\m(\Delta_H)\,d\mu\\
&=\int\lim_{r\to\infty}\frac{|\partial U_r(\Gamma,H)|}{|\partial U_r(\F_n,e)|}\,d\mu\\
&=\lim_{r\to\infty}\frac{1}{f(r)}\int|\partial U_r(\Gamma,H)|\,d\mu\\
&=\lim_{r\to\infty}\frac{1}{f(r)}\E(X_r).
\end{align*}
Passing to the subsequence $\{r+m\ell\}_{m\in\N}$ and replacing the second (and clearly smaller) term of (\ref{bound}) with the first, we see that
\begin{align*}
\lim_{r\to\infty}\frac{\E(X_r)}{f(r)}&\leq\lim_{m\to\infty}\frac{f(r)}{f(r+m\ell)}\left((2n-1)^\ell-\eps+\sqrt{\left((2n-1)^\ell-\eps\right)^2+4\eps^2}\right)^m\\
&=\lim_{m\to\infty}\left(1-\frac{\eps}{(2n-1)^\ell}+\sqrt{1-\frac{2\eps}{(2n-1)^\ell}+\frac{5\eps^2}{(2n-1)^{2\ell}}}\right)^m.
\end{align*}
But a simple calculation shows that what is inside the parentheses is less than one, so that the above limit is zero. It follows that $\E(\Delta_H)=0$ and therefore that the boundary action of our invariant random subgroup is conservative.

Suppose next that $\mu$ is a sofic random subgroup which is ergodic and does not satisfy property $D$. Then by Proposition~\ref{lop}, it has a lopsided sofic approximation, i.e.\ a sofic approximation consisting of contracted Schreier graphs $(\Gamma_i,A_i,\mu_i)$ such that $\mu_i(A_i)\to1$ and the average external degree of vertices in $A_i$ is much smaller than their average external degree, in the sense that their ratio tends to zero. Since $\mu_i(\Gamma_i\backslash A_i)\to0$, this implies that
\[
\E(\deg(x)\mid A_i)\ll\E(\deg(x)\mid \Gamma_i\backslash A_i),
\]
again in the sense that the ratio of these two quantities tends to zero. But the vertex degree of a point in a contracted Schreier graph $\Gamma^{(r)}$ is precisely one less than the size of the $r$-neighborhood of the corresponding uncontracted graph $\Gamma$. We thus find that, over a set of arbitrarily large measure, the ratio of the average size of (arbitrarily large) $r$-neighborhoods in our Schreier graphs to $|U_r(\F_n,e)|$ is arbitraily small, which proves our claim.
\end{proof}
To conclude this section, let us remark that, although Theorem~\ref{conservative} says, in effect, that sofic random subgroups cannot grow as quickly as the free group, it is reasonable to expect that they can still grow very quickly: it is proved in \cite{AGV} (see Theorem~40) that there exists a (nonatomic) regular unimodular random graph whose exponential growth rate is maximal.

\section{Cogrowth and limit sets}

It is interesting to examine other questions considered in \cite{GKN} for sofic random subgroups. Note, for example, that Theorem~\ref{conservative} immediately implies that, unless it is the Dirac measure concentrated on the $2n$-regular tree, a sofic random Schreier graph $\Gamma\in\Lambda(\F_n)$ cannot contain a \emph{branch} of $\F_n$, i.e.\ a subgraph isomorphic to the unique tree one of whose vertices has degree one and all of whose other vertices have degree $2n$, since the presence of a branch implies the existence of a nontrivial wandering set (another way to say this is that every edge of an invariant random Schreier graph must belong to a cycle). Recall, moreover, that the \emph{cogrowth} of a subgroup $H\leq\F_n$ (i.e.\ the ``growth of $H$ inside of $\F_n$'') is defined to be
\[
v_H\colonequals\limsup_{r\to\infty}\sqrt[r]{|H\cap U_r(\F_n,e)|}\leq2n-1.
\]
By Theorem~4.2 of \cite{GKN}, if $v_H<\sqrt{2n-1}$, then the action $H\circlearrowright(\partial\F_n,\m)$ is dissipative. We therefore have the following corollary of Theorem~\ref{conservative}.
\begin{corollary}\label{cogrowth}
The cogrowth of a sofic random subgroup $H\in L(\F_n)$ must satisfy $v_H\geq\sqrt{2n-1}$.
\end{corollary}
Alternatively, a Schreier graph is \emph{Ramanujan} if and only if its cogrowth does not exceed $\sqrt{2n-1}$, and it is proved in \cite{AGV} (see Theorem~5) that random unimodular $d$-regular graphs are Ramanujan if and only if they are trees, which shows that an invariant random subgroup $H$ satisfies $v_H>\sqrt{2n-1}$.

There are various limit sets associated to a subgroup $H\leq\F_n$ (most of which descend from the general theory of discrete groups of isometries of Gromov hyperbolic spaces). The \emph{radial limit set}, denoted $\Lambda_H^{\rad}$, is the set of limit points (in $\partial\F_n$) of sequences of elements of $H$ which are contained within a tubular neighborhood of a certain geodesic ray in $\F_n$. There are the \emph{small horospheric limit set}, denoted $\Lambda_H^{\hor,s}$, which is the set of boundary points $\omega\in\partial\F_n$ such that any \emph{horosphere} centered at $\omega$ contains infinitely elements of $H$, the Schreier limit set $\Omega_H$, and the \emph{big horospheric limit set}, denoted $\Lambda_H^{\hor,b}$, which is the set of boundary points $\omega\in\partial\F_n$ such that a certain horosphere centered at $\omega$ contains infinitely elements of $H$. There are also the divergence set of the \emph{Poincar\'{e} series} of $H$, denoted $\Sigma_H$, and the \emph{full limit set}, denoted $\Lambda_H$, which is the set of all limit points (in $\partial\F_n$) of elements of $H$. We refer the reader to \cite{GKN} for the precise definitions of these sets. 

As is shown in \cite{GKN}, there is a certain amount of flexibility in the $\m$-measures of the aforementioned limit sets for arbitrary subgroups $H\leq\F_n$: although several of these sets necessarily have the same measure, the measure of the full limit set $\Lambda_H$ may take on a range of values (and may well be a null set). Once again, however, the situation for sofic random subgroups is more rigid, as the following theorem shows.
\begin{theorem}\label{lim}
Let $H$ be a sofic random subgroup. Then the limit sets $\Lambda_H^{\hor,s}$, $\Omega_H$, $\Lambda_H^{\hor,b}$, $\Sigma_H$, and $\Lambda_H$ all have full $\m$-measure.
\end{theorem}
\begin{proof}
By Theorems~3.20 and 3.21 of \cite{GKN}, the aforementioned limit sets are contained in one another in the order in which we have listed them, i.e.\
\[
\Lambda_H^{\rad}\subseteq\Lambda_H^{\hor,s}\subseteq\Omega_H\subseteq\Lambda_H^{\hor,b}\subseteq\Sigma_H\subseteq\Lambda_H,
\]
and the middle four of these have the same $\m$-measure. By Theorem~\ref{conservative}, $\m(\Omega_H)=1$. These facts taken together imply the claim.
\end{proof}
By Theorem~3.35 of \cite{GKN} (which is an analogue of the \emph{Hopf-Tsuji-Sullivan theorem}, valid for discrete groups of isometries of $n$-dimensional hyperbolic space), either $\m(\Lambda_H^{\rad})=1$ or $\m(\Lambda_H^{\rad})=0$, the former occurring when the simple random walk on $(\Gamma,H)$ is recurrent and the latter when the simple random walk on $(\Gamma,H)$ is transient. The following examples show that the $\m$-measure of the radial limit set of a nonatomic invariant random subgroup may be either zero or one.
\begin{example}(An invariant random subgroup with the property that $\m(\Lambda_H^{\rad})=1$) Consider the Cayley graph $\Gamma$ of the group $\Z^2$ constructed with respect to the standard generators $a=(1,0)$ and $b=(0,1)$. It is a classical result that the simple random walk on $\Z^2$ is recurrent \cite{Po}, so Theorem~3.35 of \cite{GKN} implies that $\m(\Lambda_H^{\rad})=1$, where $H$ is the fundamental group of $\Gamma$. The graph $\Gamma$ contains infinitely many ``$a$-chains,'' i.e.\ bi-infinite geodesics labeled with the generator $a$, and by independently reversing the orientations of these $a$-chains or leaving their orientations fixed, we generate a large space of Schreier graphs each of whose underlying unlabeled graphs is isomorphic to the two-dimensional integer lattice (in particular, the simple random walk on these graphs remains recurrent). There is natural uniform measure on this space (the uniform measure on its projective structure), and it is not difficut to see that this measure is invariant.
\end{example}
\begin{example}(An invariant random subgroup with the property that $\m(\Lambda_H^{\rad})=0$) Consider the Cayley graph $\Gamma$ of the group $\Z^3$ constructed with respect to the standard generators $a=(1,0,0)$, $b=(0,1,0)$, and $c=(0,0,1)$. It is again a classical result that the simple random walk on $\Z^3$ (or, indeed, on $\Z^n$ for $n\geq3$) is transient, so that $\m(\Lambda_H^{\rad})=0$, where $H$ is the fundamental group of $\Gamma$. By employing the same trick as in the previous example, we again generate a large space of Schreier graphs for which the uniform measure is a nonatomic invariant probability measure.
\end{example}

\end{document}